\documentclass[
            a4paper,
            11pt,
            reqno,
            centertags
                ]{amsart}

\usepackage[utf8]{inputenc}
\usepackage[ngerman,english]{babel}
\usepackage[T1]{fontenc}
\usepackage{lmodern,stackengine}

\usepackage[
        spacing    = true, %see microtype doc
        kerning    = true, %optimise space around paired delims
        tracking   = true, %spacing around SC
        expansion  = true, %see microtype doc
        protrusion = true  %margin breakthrough for small chars
           ]{microtype}

\usepackage{amssymb,amsthm,bbm,bm}
\usepackage[fixamsmath,disallowspaces]{mathtools}

\newtheoremstyle{ToBeProved}
  {.5\baselineskip}
  {.5\baselineskip}
  {\itshape}
  {}
  {\bfseries}
  {.}
  {0.5em}
  {}

\newtheoremstyle{BelieveIt}
  {.5\baselineskip}
  {.5\baselineskip}
  {}
  {}
  {\bfseries}
  {.}
  {0.5em}
  {}

\theoremstyle{ToBeProved}
\newtheorem*{thm*}{Theorem}
\newtheorem{thm}{Theorem}[section]
\newtheorem{lem}[thm]{Lemma}
\newtheorem{cor}[thm]{Corollary}
\newtheorem{pro}[thm]{Proposition}

\theoremstyle{BelieveIt}
\newtheorem{defi}[thm]{Definition}

\newtheorem{rem}[thm]{Remark}

\newcommand{\defgl}{\mathrel{\mathop:\!\!=}} % A := B
\newcommand{\gldef}{\mathrel{=\!\!\mathop:}} % A =: B
\newcommand{\longto}{\longrightarrow}
\newcommand{\refpnt}{\,|\,}
\newcommand{\letter}[1]{\mathtt{#1}}
\newcommand{\mathbold}[1]{\bm{#1}}
\newcommand{\ts}{\hspace{0.5pt}}
\newcommand{\II}{\ts\mathrm{i}\ts}
\newcommand{\remend}{\hfill\small{$\blacklozenge$}}
\newcommand{\dif}{\mathop{}\!\mathrm{d}}
\newcommand{\IC}{\mathbb{C}}
\newcommand{\IE}{\mathbb{E}}
\newcommand{\IN}{\mathbb{N}}
\newcommand{\IR}{\mathbb{R}}
\newcommand{\IV}{\mathbb{V}}
\newcommand{\IX}{\mathbb{X}}
\newcommand{\IY}{\mathbb{Y}}
\newcommand{\IZ}{\mathbb{Z}}
\newcommand{\CA}{\mathcal{A}}
\newcommand{\CC}{\mathcal{C}}
\newcommand{\CD}{\mathcal{D}}
\newcommand{\CG}{\mathcal{G}}
\newcommand{\CL}{\mathcal{L}}
\newcommand{\CN}{\mathcal{N}}
\newcommand{\CP}{\mathcal{P}}
\newcommand{\CS}{\mathcal{S}}
\newcommand{\CZ}{\mathcal{Z}}
\newcommand{\FB}{\mathfrak{B}}
\newcommand{\FZ}{\mathfrak{Z}}

\newcommand{\subw}{\vartriangleleft}
\DeclareMathOperator{\steq}{\,\topinset{{\tiny 
\scalebox{0.75}{$\bullet$}}}{=}{-3.5pt}{2.5pt}\,}
\def\stackalignment{r}
\DeclareMathOperator{\stsub}{\hspace{1pt}\topinset{{\tiny 
\scalebox{0.75}{$\bullet$}}}{$\vartriangleleft$}{2.15pt}{1.55pt}\,}

\renewcommand{\le}{\leqslant}
\renewcommand{\ge}{\geqslant}

\DeclareMathOperator{\mat}{Mat}
\DeclareMathOperator{\Rea}{Re}
\DeclareMathOperator{\sinc}{sinc}
\DeclareMathOperator{\dens}{dens}
\DeclareMathOperator{\vol}{vol}
\DeclareMathOperator{\e}{e}

\usepackage{tikz}
    \usetikzlibrary{matrix,arrows,decorations.markings,decorations.pathreplacing}

\usepackage[labelfont=bf]{caption}
\usepackage[shortlabels]{enumitem}
\usepackage[
            linktocpage = true,
            pdfauthor   = {Markus Moll},
            pdftitle    = {Diffraction Of Random Noble Means Words},
            colorlinks  = true,
            linkcolor   = black,
            citecolor   = black,
            urlcolor    = black,
            breaklinks  = true
                ]{hyperref}

\begin{document}
    \linespread{1.08}
	\title{Diffraction of Random Noble Means Words}
	\author{Markus Moll}
    \email{mmoll@math.uni-bielefeld.de}
    \address{University of Bielefeld, Faculty of Mathematics, Universitätsstr.
    25, D--33615 Bielefeld, Germany}
	\begin{abstract}
        In this paper, several aspects of the random noble means substitution
        are studied. Beyond important dynamical facets as the frequency of
        subwords and the computation of the topological entropy, the important
        issue of ergodicity is addressed. From the geometrical point of view,
        we outline a suitable cut and project setting for associated point sets
        and present results for the spectral analysis of the diffraction
        measure.
	\end{abstract}
	\maketitle
    
    \section{Introduction}
    
    In 1989, Godr{\`e}che and Luck \cite{luck} introduced a (locally)
    randomised extension of the well-studied Fibonacci substitution. They
    presented first results concerning the topological entropy and the spectral
    type of the diffraction measure. In this context, it is most remarkable
    that the dynamical hull features positive entropy but at the same time is
    regular enough to contain only Meyer sets. The arguments applied in
    \cite[Sec.~5.1]{luck} for the computation of the topological entropy rely
    on the fact that it is sufficient to merely control the growth behaviour of
    exact random Fibonacci words. This is a non-trivial assertion and has only
    recently been proved by Nilsson \cite{nilsson} via intricate combinatorial
    arguments. Furthermore, Godr{\`e}che and Luck argued via a concrete
    calculation that the diffraction measure comprises a continuous part.
    There, they implicitly assumed the existence of an ergodic measure on the
    randomised hull without proof or other evidence.
    
    In this paper, we will generalise the random Fibonacci substitution to the
    one-parameter family of random noble means substitutions and substantiate
    the results of Godr{\`e}che and Luck with mathematical rigour.
	
	\section{Notation}
	
    Let us start with a brief summary of the essential notation that will be
    used throughout the text. We will loosely complement this list as we
    continue. A more detailed introduction can be found in standard textbooks;
    see \cite{bagrimm,fogg,kitchens,lothaire}.
    
    The finite \emph{alphabet} on $n$ letters is denoted by $\CA_{n}^{} \defgl
    \{ \letter{a}_{i}^{} \mid 1 \le i \le n \}$ and we refer to
    $\CA_{n}^{\ast}$ as the free monoid over $\CA_{n}^{}$. The latter is the
    set of finite words over $\CA_{n}^{}$ together with the empty word
    $\varepsilon$ and endowed with the concatenation of words as
    multiplication. Let $v$, $w \in \CA_{n}^{\ast}$ and $v$ be a connected
    substring of $w$. Then, we call $v$ a \emph{subword} of $w$ and write $v
    \subw w$ in this case. If a more precise emphasis on the location of a
    subword is needed, we will write $w_{[j,k]} \defgl w_{j}^{} \cdots w_{k}^{}
    \subw w$ where $w_{[j,k]} \defgl \varepsilon$ if $j > k$. The \emph{length}
    of some word $w \in \CA_{n}^{\ast}$ will be written as $\lvert w \rvert$
    and $\lvert w \rvert_{v}^{} = \lvert \{ k \mid v = w_{[k,k+\lvert v \rvert
    - 1]} \} \rvert$ is the \emph{occurrence number} of the word $v \in
    \CA_{n}^{\ast}$ in $w$ as a subword. The set $\CA_{n}^{\IZ}$ of bi-infinite
    sequences over $\CA_{n}^{}$ is equipped with the product topology that is
    assumed to be generated by the class $\mathfrak{Z}(\CA_{n}^{\IZ})$ of
    cylinder sets
    \begin{equation*}
        \CZ_{k}^{}(v) \defgl
        \bigl\{
            w \in \CA_{n}^{\IZ} \mid w_{[k,k+ \lvert v \rvert - 1]}^{} = v
        \bigr\},
    \end{equation*}
    for any $k \in \IZ$ and $v \in \CA_{n}^{\ast}$, and for the purpose of our
    considerations it will be convenient to regard $\CA_{n}^{\ast}$ as being
    embedded into $\CA_{n}^{\IZ}$.

    A \emph{substitution rule} $\vartheta$ is any non-erasing endomorphism on
    $\CA_{n}^{\ast}$ that can and will be extended to $\CA_{n}^{\IZ}$ via
    concatenation.

    \section{The random noble means substitution}

    For the rest of the treatment, we fix the binary alphabet $\CA_{2}^{} = \{
    \letter{a},\letter{b} \}$, an arbitrary $m \in \IN$ and define for each
    $0 \le i \le m$ a \emph{noble means substitution (NMS)}
    $\zeta_{m,i}^{}$ on $\CA_{2}^{\IZ}$ via
	\begin{equation*}
	    \zeta_{m,i}^{} \colon
	    \biggl\{
	        \begin{array}{lll}
                \letter{a}  & \longmapsto & \letter{a}^{i} \letter{b}
                \letter{a}^{m - i}, \\
	            \letter{b}  & \longmapsto & \letter{a},
	        \end{array}
	    \biggr.
	    \quad \text{where} \quad M_{m}^{} \defgl M_{\zeta_{m,i}}^{} \defgl
	    \begin{pmatrix*}
	        m & 1 \\
	        1 & 0
	    \end{pmatrix*}
	\end{equation*}
    is its primitive and unimodular substitution matrix that is independent of
    $i$. Its Perron--Frobenius (PF) eigenvalue \cite{seneta} is the
    Pisot--Vijayaraghavan (PV) number $\lambda_{m}^{} \defgl (m + \sqrt{m^{2} +
    4})/2$ which has algebraic conjugate $\lambda_{m}^{\prime} = m
    -\lambda_{m}^{}$. The discrete hull $\IX_{m,i}^{}$ of each $\zeta_{m,i}^{}$ is
    defined as the orbit closure of some fixed point of a suitable power of
    $\zeta_{m,i}^{}$, with respect to the shift $S$, in the product topology.  Now,
    one convenient property of the noble means family $\CN_{m}^{} \defgl \{
    \zeta_{m,i}^{} \mid 0 \le i \le m \}$ is that all these hulls coincide
    individually which is a direct consequence of the primitivity of each
    $\zeta_{m,i}^{}$ and the fact that all $\zeta_{m,i}^{}$ are pairwise
    conjugate \cite[Prop.~4.6]{bagrimm}. As our final goal is the local mixture
    of all members of $\CN_{m}^{}$, this constitutes a substantial technical
    simplification over the more general situation. Several important
    properties of the NMS family can be summarised as follows; compare
    \cite[Lem.~2.9]{moll}.
	
    \begin{lem} \label{lem:det_noble_means}
        For an arbitrary but fixed $m \in \IN$, each member of $\CN_{m}^{}$ is
        a primitive and aperiodic Pisot substitution with unimodular
        substitution matrix. Its two-sided discrete hulls $\IX_{m,i}^{}$ are
        uncountable and reflection symmetric, and the $\IX_{m,i}^{}$ coincide
        for $0 \le i \le m$.  \qed
	\end{lem}

    We proceed with the general notion of a random substitution rule. Note that
    the mixture is performed on a local level i.e. the image of each letter of
    some word under the substitution rule is chosen seperately and
    independently. In the noble means case the locality leads to a significant
    enlargement of the according discrete hull whereas the hull would stay the
    same when studying global mixtures of the substitutions in $\CN_{m}^{}$.
    This is an immediate consequence of Lemma~\ref{lem:det_noble_means}.
	
	\begin{defi} \label{def:random_subst}
        A substitution $\vartheta \colon \CA_{n}^{\ast} \longto \CA_{n}^{\ast}$
        is called \emph{stochastic} or a \emph{random substitution} if there
        are $k_{1}^{},\ldots, k_{n}^{} \in \IN$ and probability vectors
	    \begin{equation*}
	        \Bigl\{
	            \mathbold{p}_{i}^{} = (p_{i1}^{},\ldots, p_{ik_{i}}^{}) \mid
	            \mathbold{p}_{i}^{} \in [0,1]^{k_{i}}
	            \text{ and }
	            \sum_{j=1}^{k_{i}} p_{ij}^{} = 1, \, 1 \le i \le n
	        \Bigr\},
	    \end{equation*}
	    such that
	    \begin{equation*}
	        \vartheta \colon \letter{a}_{i}^{} \longmapsto
	        \left\{
	            \begin{array}{cc}
	                w^{(i,1)},     & \text{with probability } p_{i1}^{},     \\
	                \vdots         & \vdots                                  \\
	                w^{(i,k_{i})}, & \text{with probability } p_{ik_{i}}^{}, \\
	            \end{array}
	        \right.
	    \end{equation*}
        for $1 \le i \le n$ where each $w^{(i,j)} \in \CA_{n}^{\ast}$. The
        \emph{substitution matrix} is defined by 
	    \begin{equation*}
	        M_{\vartheta}^{} \defgl
	        \Bigl(
	            \sum_{q=1}^{k_{j}} p_{jq}^{}\lvert w^{(j,q)} 
	            \rvert_{\letter{a}_{i}^{}}^{} 
	        \Bigr)_{ij}^{} \in \mat(n,\IZ).
	    \end{equation*}
	\end{defi}
	
	\begin{rem} 
        In the stochastic situation we agree on a slightly modified notion of
        the subword relation.  For any $v$, $w \in \CA_{n}^{\ast}$, by $v
        \stsub \vartheta^{k}(w)$  we mean that $v$ is a subword of \emph{at
        least} one image of $w$ under $\vartheta^{k}$ for any $k \in \IN$.
        Similarly, by $v \steq \vartheta^{k}(w)$ we mean that there is \emph{at
        least} one image of $w$ under $\vartheta^{k}$ that coincides with $v$.
        \remend
	\end{rem}
	
	\begin{defi}
        A random substitution $\vartheta \colon \CA_{n}^{\ast} \longto
        \CA_{n}^{\ast}$ is \emph{irreducible} if for each pair $(i,j)$ with $1
        \le i,j \le n$, there is a power $k \in \IN$ such that
        $\letter{a}_{i}^{} \stsub \vartheta^{k}(\letter{a}_{j}^{})$. The
        substitution $\vartheta$ is \emph{primitive} if there is a $k \in \IN$
        such that $\letter{a}_{i}^{} \stsub \vartheta^{k}(\letter{a}_{j}^{})$
        for all $1 \le i,j \le n$.
	\end{defi}
	
    Now, let $m \in \IN$ and $\mathbold{p}_{m}^{} = (p_{0}^{},\ldots,
    p_{m}^{})$ be a probability vector that are both assumed to be fixed.  That
    means $\mathbold{p}_{m}^{} \in [0,1]^{m+1}$  and $\sum_{j = 0}^{m} p_{j}^{}
    = 1$. The random substitution $\zeta_{m}^{} \colon \CA_{2}^{\ast} \longto
    \CA_{2}^{\ast}$ is defined by 
	\begin{equation}
	    \zeta_{m}^{} \colon 
	    \left\{
            \begin{array}{lll}
                \letter{a} & \longmapsto & \left\{
                \begin{array}{cc}
                    \zeta_{m,0}^{}(\letter{a}), & \text{with probability } p_{0}^{}, \\
                    \vdots                      & \vdots                             \\
                    \zeta_{m,m}^{}(\letter{a}), & \text{with probability } p_{m}^{},
                \end{array}\right. \\
                \letter{b} & \longmapsto & \letter{a},
            \end{array}
	    \right.
	\end{equation}
    and the one-parameter family $\mathcal{R} = \{ \zeta_{m}^{} \}_{m \in
    \IN}^{}$ is called the family of \emph{random noble means substitutions
    (RNMS)}.  We refer to the $p_{j}^{}$ as the \emph{choosing probabilities}
    and call $\zeta_{m}^{}(w)$ for any $w \in \CA_{2}^{\ast}$ an \emph{image}
    of $w$ under $\zeta_{m}^{}$. Of course, the deterministic cases of the
    family $\CN_{m}^{}$ (choose the corresponding $p_{j}^{} = 1$) and
    \emph{incomplete} mixtures, with several $p_{j}^{} = 0$, are included here
    but we are mainly interested in the generic cases where
    $\mathbold{p}_{m}^{} \gg 0$. This is a standing assumption for the rest of
    the treatment, where we occasionally comment on the disregarded cases if
    this seems appropriate. The substitution matrix of $\zeta_{m}^{}$ in the
    sense of Definition~\ref{def:random_subst} is given by
	\begin{equation*}
	    M_{m}^{} \defgl 
	    \begin{pmatrix*}
            \sum_{j=0}^{m} p_{j}^{} \lvert \zeta_{m,j}^{}(\letter{a})
            \rvert_{\letter{a}}^{} & 1 \\
            \sum_{j=0}^{m} p_{j}^{} \lvert \zeta_{m,j}^{}(\letter{a})
            \rvert_{\letter{b}}^{} & 0
	    \end{pmatrix*}
	    =
	    \begin{pmatrix*}
	        m & 1 \\
	        1 & 0
	    \end{pmatrix*}.
	\end{equation*}
    Due to the fact that there is no direct analogue to a bi-infinite fixed
    point in the randomised case, we have to slightly modify the notion of the
    discrete hull here.
    \begin{defi}
        For an arbitrary but fixed $m \in \IN$, define
        \begin{equation*}
            X_{m}^{} \defgl 
            \Bigl\{ 
                w \in \mathcal{A}_{2}^{\mathbb{Z}} \mid w
                \text{ is an accumulation point of } \bigl( \zeta_{m}^{k} 
                (\letter{a} \refpnt \letter{a}) \bigr)_{k \in \mathbb{N}_{0}}
            \Bigr\}.
        \end{equation*}
        The \emph{two-sided discrete stochastic hull} $\mathbb{X}_{m}^{}$ is
        defined as the smallest closed and shift-invariant subset of
        $\mathcal{A}_{2}^{\mathbb{Z}}$ with $X_{m}^{} \subset
        \mathbb{X}_{m}^{}$. Elements of $X_{m}^{}$ are called \emph{generating
        random noble means words.}
    \end{defi}
    A word $w \in \CA_{2}^{\ast}$ is called \emph{legal} (or
    $\zeta_{m}^{}$-\emph{legal}) if there is a $k \in \IN$ such that $w \stsub
    \zeta_{m}^{k}(\letter{b})$. For $\ell \ge 0$, we define
    \begin{equation*}
        \CD_{m}^{} \defgl 
        \bigl\{
            w \in \CA_{2}^{\ast} \mid w \text{ is } \zeta_{m}^{} \text{-legal}
        \bigr\} 
        \quad \text{and} \quad \CD_{m,\ell}^{} \defgl 
        \bigl\{
            w \in \CD_{m}^{} \mid \lvert w \rvert = \ell
        \bigr\}.
    \end{equation*}
    If $w \steq \zeta_{m}^{k}(\letter{b})$ for some $k \in \IN_{0}^{}$, we
    refer to $w$ as an \emph{exact substitution word} and define for any $k \ge
    1$ the set of exact substitution words (of order $k$) as
    \begin{equation*}
        \CG_{m,k}^{} \defgl 
        \bigl\{ 
            w \in \CA_{2}^{\ast} \mid w \steq \zeta_{m}^{k-1}(\letter{b}) 
        \bigr\}.
    \end{equation*}
    A convenient approach to the set of exact RNMS words is the following 
    concatenation rule. For $k \ge 3$, let
    \begin{equation} \label{equ:concat_rule}
        \CG_{m,k}^{} \defgl \bigcup_{i=0}^{m} \prod_{j=0}^{m}
        \CG_{m,k-1-\delta_{ij}^{}}^{}
        \quad \text{with}\quad
        \CG_{m,1}^{} \defgl \{ \letter{b} \}
        \quad \text{and} \quad
        \CG_{m,2}^{} \defgl \{ \letter{a} \},
    \end{equation}
    where $\delta_{ij}^{}$ denotes the \emph{Kronecker function}. The product
    in Eq.~$\eqref{equ:concat_rule}$ is understood via the concatenation of
    words and each word $w \in \CG_{m,k}^{}$ is of length $\ell_{m,k}^{} \defgl
    m\ell_{m,k-1}^{} + \ell_{m,k-2}$ with $\ell_{m,1}^{} \defgl 1 \gldef
    \ell_{m,2}^{}$. Obviously, not all legal words of length $\ell_{m,k}^{}$
    are exact (e.g. $\letter{aa}$, $\letter{bb} \in \CD_{m,2}^{} \setminus
    \CG_{m,3}^{}$).
    \begin{figure}
       \centering
       \includegraphics[scale=0.6]{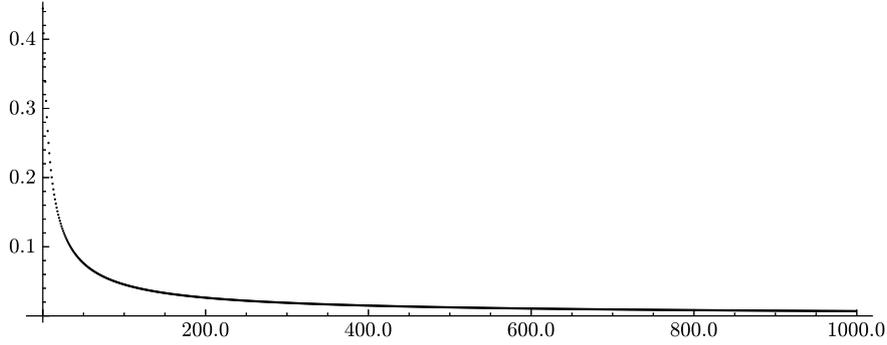}
       \caption{The topological entropy $\mathcal{H}_{m}^{}$ for $1 \le m \le
           1000$.}
        \label{fig:entropy}
    \end{figure}
    
    The set of exact RNMS words facilitates a convenient method for the
    computation of the topological entropy. Applying a theorem of Nilsson
    \cite[Thm.~3]{nilsson} and carrying out a short calculation, concerning the
    cardinalities of exact RNMS sets, yields the following result
    \cite[Sec.~3.2]{moll} for the topological entropy $\mathcal{H}_{m}^{}$ in
    the RNMS case.
    \begin{align*}
        \mathcal{H}_{m}^{} &= \lim_{k \to \infty} \frac{\log \bigl(\lvert
        \CG_{m,k}^{} \rvert \bigr)}{\ell_{m,k}^{}} = \frac{\lambda_{m}^{} -
        1}{1 - \lambda_{m}^{\prime}} \sum_{i = 2}^{\infty} \frac{\log \bigl(m(i
        - 1) + 1\bigr)}{\lambda_{m}^{i}},
    \end{align*}
    which is strictly positive. This is in contrast to the deterministic cases
    of $\CN_{m}^{}$ where each element of $\IX_{m,i}^{}$ is a \emph{Sturmian
    sequence} \cite[Prop.~3.2]{moll} which means that the topological entropy
    vanishes here.
	
	\section{Ergodicity}
	
    In this section, we define a shift-invariant probability measure on the
    discrete RNMS hull $\IX_{m}^{}$ and prove its ergodicity. The result is
    somewhat weaker as in all deterministic cases of $\CN_{m}^{}$, because it
    is known that the hulls of primitive substitutions are \emph{minimal} and
    that there is a \emph{uniquely} ergodic probability measure \cite{queff}.
    As $\IX_{m,i}^{} \subsetneq \IX_{m}^{}$ \cite[Prop.~2.22]{moll}, one
    directly observes the non-minimality of $\IX_{m}^{}$ and the non-uniqueness
    of the measure can be expected immediately and will be proved explicitly
    later.
    \begin{defi} \label{def:induced_subst}
        Let $\ell \in \IN$ and $\zeta_{m}^{} \colon \CA_{2}^{\ast} \longto
        \CA_{2}^{\ast}$ be a random noble means substitution for some fixed $m
        \in \IN$. Then, we refer to
        \begin{equation*}
            ( \zeta_{m}^{} )_{\ell}^{} \colon \CD_{m,\ell}^{\ast} \longto
            \CD_{m,\ell}^{\ast}
        \end{equation*}
        as the \emph{induced substitution} defined by
        \begin{equation*}
            ( \zeta_{m}^{} )_{\ell}^{} \colon w^{(i)} \longmapsto
            \left\{
            \begin{array}{cc}
                u^{(i,1)} \defgl \Bigl( v^{(i,1)}_{[k,k + \ell - 1]} \Bigr)_{0 \le
                k \le \lvert \zeta_{m}^{}(w_{0}^{(i)}) \rvert - 1 },
                &\text{with prob. } p_{i1}^{}, \\
                \vdots & \vdots \\
                u^{(i,n_{i})} \defgl \Bigl( v^{(i,n_i)}_{[k,k + \ell - 1]}
                \Bigr)_{0 \le k \le \lvert \zeta_{m}^{}(w_{0}^{(i)}) \rvert - 1},
                &\text{with prob. } p_{in_{i}^{}}^{}, \\
            \end{array}
            \right.
        \end{equation*}
        where $w^{(i)} \in \CD_{m,\ell}^{}$ and $v^{(i,j)} \in \CD_{m}^{}$ is
        an image of $w^{(i)}$ under $\zeta_{m}^{}$ with probability
        $p_{ij}^{}$.
    \end{defi}
    One can show that the induced substitution matrix $M_{m,\ell}^{}$ of
    $(\zeta_{m}^{})_{\ell}^{}$ is primitive \cite[Prop.~4.7]{moll} which
    enables the reapplication of Perron--Frobenius theory. Note that
    $(\zeta_{m}^{})_{1}^{} = \zeta_{m}^{}$ and therefore $M_{m,1}^{} =
    M_{m}^{}$. In the case of $\ell = 2$, one can explicitly work out
    $M_{m,2}^{}$ for arbitrary $m \in \IN$ \cite[Prop.~4.10]{moll} and proceed
    recursively for the generalisation to any word length $\ell \in \IN$
    \cite[Cor.~4.13]{moll}. One finds
    \begin{equation*}
        M_{m,2}^{} =
        \begin{pmatrix*}
            m-1+p_{0}^{}p_{m}^{}   & m-1+p_{0}^{}   & 1-p_{0}^{} & 1 \\
            1-p_{0}^{}p_{m}^{}     & 1-p_{0}^{}     & p_{0}^{}   & 0 \\
            1-p_{0}^{}p_{m}^{}     & 1              & 0          & 0 \\
            p_{0}^{}p_{m}^{}       & 0              & 0          & 0
        \end{pmatrix*},
    \end{equation*}
    with statistically normalised right PF eigenvector
    \begin{equation} \label{equ:induced_right_pf}
        \mathbold{R}_{m,2}^{} =
        \begin{pmatrix*}
            \frac{2(\lambda_{m}^{} - 1)}{m(1+p_{0}^{} p_{m}^{}) - (2 +
            2\lambda_{m}^{} - m)(-1 + p_{0}^{} p_{m}^{})} \\
            \frac{2(1 - p_{0}^{} p_{m}^{})}{m(1+p_{0}^{} p_{m}^{}) - (2 +
            2\lambda_{m}^{} - m)(-1 + p_{0}^{} p_{m}^{})}    \\
            \frac{2(1 - p_{0}^{} p_{m}^{})}{m(1+p_{0}^{} p_{m}^{}) - (2 +
            2\lambda_{m}^{} - m)(-1 + p_{0}^{} p_{m}^{})}    \\
            \frac{2(1 + \lambda_{m}^{\prime}) p_{0}^{} p_{m}^{}}{m(1+p_{0}^{}
            p_{m}^{}) - (2 + 2\lambda_{m}^{} - m)(-1 + p_{0}^{} p_{m}^{})}
        \end{pmatrix*}.
    \end{equation}
    Now, let $w \in \CD_{m,\ell}^{}$ be any
    $\zeta_{m}^{}$-legal word. Then, we define the measure $\mu_{m}^{}$ on the
    cylinder sets $\CZ_{k}^{}(w)$ by
    \begin{equation} \label{equ:ergodic_meas}
        \mu_{m}^{}
        \bigl( 
            \CZ_{k}^{}(w)
        \bigr) \defgl \mathbold{R}_{m,\ell}^{}(w),
    \end{equation}
    for any $k \in \IZ$, where $\mathbold{R}_{m,\ell}^{}(w)$ is the entry of
    the statistically normalised right PF eigenvector of $M_{m,\ell}^{}$ with
    respect to the word $w$. According to \cite[Sec.~5.4]{queff}, this is a
    consistent definition of a measure on $\FZ(\IX_{m}^{})$ and there is an
    extension of $\mu_{m}^{}$ to the Borel $\sigma$-algebra $\FB_{m}^{}$
    \cite[Cor.~2.4.9]{partha} generated by the cylinder sets. Due to
    \cite[Prop.~2.5.1]{partha}, this extension is unique and we will denote it
    again as $\mu_{m}^{}$. Note that Eq.~$\eqref{equ:induced_right_pf}$
    indicates that $\mu_{m}^{}$ depends on the choice of $\mathbold{p}_{m}^{}$,
    whereas the hull $\IX_{m}^{}$ is invariant under alterations of the
    choosing probabilities as long as $\mathbold{p}_{m}^{} \gg 0$. The same is
    true for any $\ell \in \IN$ which means that there are infinitely many
    possibilities to construct a probability measure for the very same
    $\IX_{m}^{}$ in the above way. We proceed with an important ingredient for
    the proof of the ergodicity of $\mu_{m}^{}$.
    
    \begin{thm}[{\cite[Thm.~1]{etemadi}}]
        \label{thm:slln}
        Let $(X_{i}^{})_{i \in \IN}^{}$ be a family of pairwise independent,
        identically distributed, complex random variables with common
        distribution $\mu$, subject to the integrability condition
        $\IE_{\mu}^{}(\lvert X_{1}^{} \rvert) < \infty$. Then,
        \begin{equation*}
            \frac{1}{n} \sum_{i = 1}^{n} X_{i}^{} \;\xrightarrow[\text{a.s.}]{n
            \to \infty}\; \IE_{\mu}^{}(X_{1}^{}) = \int_{\IR} x \dif \mu(x).
            \tag*{\qed}
        \end{equation*}
    \end{thm}
    
    Here, $\IE_{\mu}^{}(X)$ denotes the \emph{mean} of the random variable $X$
    with respect to the distribution $\mu$.
    
	\begin{pro} \label{pro:ergodic}
        For an arbitrary but fixed $m \in \IN$, let $\IX_{m}^{} \subset
        \CA_{2}^{\IZ}$ be the two-sided discrete stochastic hull of the random
        noble means substitution and $\mu_{m}^{}$ be the $S$-invariant
        probability measure on $\IX_{m}^{}$ introduced in
        Eq.~$\eqref{equ:ergodic_meas}$. For any $f \in
        L^{1}(\IX_{m}^{},\mu_{m}^{})$ and for an arbitrary but fixed $s \in
        \IZ$, the identity
	    \begin{equation} \label{equ:ergodic}
            \lim_{n \to \infty} \frac{1}{n} \sum_{i = s}^{n + s - 1} f(S^{i}x)
            = \int_{\IX_{m}^{}}f \dif \mu_{m}^{}
	    \end{equation}
	    holds for $\mu_{m}^{}$-almost every $x \in \IX_{m}^{}$.
	\end{pro}
	
	\begin{proof}
        Let $x \in \IX_{m}^{}$ be an arbitrary element of the stochastic hull.
        The idea is to consider the characteristic function
        $\mathbbm{1}_{\CZ}^{}$ of some cylinder set $\CZ \in \FZ(\IX_{m}^{})$
        and to interpret $X \defgl \bigl( \mathbbm{1}_{\CZ}^{}(S^{i}x)
        \bigr)_{i \in \IN}^{}$ as a family of $\mu_{m}^{}$-distributed random
        variables in order to invoke Theorem~\ref{thm:slln}. For this purpose,
        we have to deal with the pairwise independence of elements in $X$. One
        can show that there is at least one element $x^{\prime} \in \IX_{m}^{}$
        with $\zeta_{m}^{}(x^{\prime}) \steq x$ \cite[Rem.~2.25]{moll} which
        means that we can study the structure of $x$ that is induced by the
        action of $\zeta_{m}^{}$ on some element of $\IX_{m}^{}$. For two
        finite subwords $u$, $v \in \CD_{m,\ell}^{}$ of $x$, we denote by $u
        \Cap v$ the overlap of $u$ and $v$ in $x$ and by $\lvert u \Cap v
        \rvert$ its number of letters. Certainly, $u$ and $v$ cannot be
        independent if $\lvert u \Cap v \rvert > 0$, but we have to take more
        into account. Possibly, $u$ and $v$ may contain parts of the image of
        the same letter under $\zeta_{m}^{}$. As $\lvert
        \zeta_{m}^{}(\letter{a}) \rvert = m + 1 > 1 = \lvert
        \zeta_{m}^{}(\letter{b}) \rvert$, it is sufficient to ensure that at
        most one of the overlaps $u \Cap \zeta_{m}^{}(\letter{a})$ and $v \Cap
        \zeta_{m}^{}(\letter{a})$ is non-empty for the very same letter
        $\letter{a} \subw x^{\prime}$, as illustrated in
        Figure~\ref{fig:independence_overlap}.
	    \begin{figure}[t]
	        \centering
	        \begin{tikzpicture}[scale = 0.85]
	        \node[] (1) at (-5.0, 1.6) {};
	        \node[] (2) at ( 5.0, 2.0) {};
	        \node[] (3) at (-1.0, 1.6) {};
	        \node[] (4) at (-1.0, 2.0) {};
	        \node[] (5) at ( 0.0, 2.0) {};
	        \node[] (6) at ( 0.0, 1.6) {};
	        \node[] (7) at (-5.0, 0.0) {};
	        \node[] (8) at ( 5.0, 0.4) {};
	        \node[] (9) at (-4.0, 0.0) {};
	        \node[](10) at (-4.0, 0.4) {};
	        \node[](11) at (-1.0, 0.4) {};
	        \node[](12) at (-1.0, 0.0) {};
	        \node[](13) at ( 1.0, 0.0) {};
	        \node[](14) at ( 1.0, 0.4) {};
	        \node[](15) at ( 4.0, 0.0) {};
	        \node[](16) at ( 4.0, 0.4) {};
	        \node[thick](18) at (-5.5, 1.8) {$\ldots$};
	        \node[thick](19) at (-5.5, 0.2) {$\ldots$};
	        \node[thick](20) at ( 5.5, 1.8) {$\ldots$};
	        \node[thick](21) at ( 5.5, 0.2) {$\ldots$};
	        \fill[lightgray] (1) rectangle (2);
	        \fill[black]     (1) rectangle (4);
	        \fill[black]     (6) rectangle (2);
	        \fill[lightgray] (7) rectangle (8);
	        \fill[black]     (7) rectangle (10);
	        \fill[black]    (12) rectangle (14);
	        \fill[black]    (15) rectangle (8);
	        \fill[lightgray] (-1.0, 1.6) -- (-1.5, 0.4) -- ( 0.5, 0.4) -- ( 0.0, 1.6)
	        -- (-1.0, 1.6);
	        \node[](22) at (-0.5, 1.8) {$\letter{a}$};
	        \node[](23) at (-6.5, 1.8) {};
	        \node[](24) at (-6.5, 0.2) {};
	        \node[](25) at (-2.5, 0.2) {$u$};
	        \node[](25) at ( 2.5, 0.2) {$v$};
	        \node[](26) at ( 0.0,-0.7) {$m$};
	        \node[](27) at ( 6.45, 1.85) {$x^{\prime}$};
	        \node[](27) at ( 6.6, 0.2) {$\zeta_{m}^{}(x^{\prime})$};
	        \draw[thick,decorate,decoration={brace,amplitude=7pt}] ( 1.0,-0.1) -- 
	        (-1.0, -0.1);
	        \path[thick,-stealth] (23) edge node[left] {$\zeta_{m}^{}$} (24);
	        \end{tikzpicture}
	        \caption[Independence of words after a sufficiently large shift.]{The
                words $u$, $v \in \CD_{m,\ell}^{}$ are independent as of the shift by
                $\ell + m$ positions. The word $\zeta_{m}^{}(\letter{a})$ can have
                non-empty overlap with precisely one of the two words.}
            \label{fig:independence_overlap}
	    \end{figure}
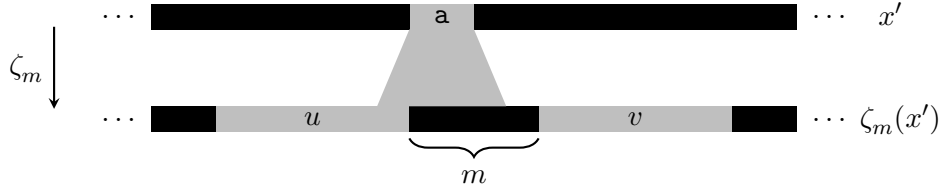
        Now, define for any $i \in \IZ$, $\ell \in \IN$ and a fixed $t \in
        \IZ$, the family
	    \begin{equation*}
	        (X_{i,k}^{})_{k \in \IN_{0}^{}}^{} \defgl
	        \Bigl(
                \bigl( 
                    S^{i+k(\ell + m)}x 
                \bigr)_{[t,t + \ell - 1]}
	        \Bigr)_{k \in \IN_{0}^{}}^{}.
	    \end{equation*}
        Then, each $X \in \bigl\{ (X_{i,k}^{})_{k \in \IN_{0}}^{} \mid s \le i
        \le \ell + m + s - 1 \bigr\}$ consists of pairwise independent words in the
        sense pointed out above. Furthermore, for any $v \in \CD_{m,\ell}^{}$, we
        consider the characteristic function of the cylinder set
        $\CZ_{t}^{}(v)~\in~ \FZ(\IX_{m}^{})$, defined by
	    \begin{equation*} 
	        \mathbbm{1}_{\CZ_{t}^{}(v)}^{}(x) \defgl
	        \left\{
	            \begin{array}{ll}
	                1, &\text{if } x_{[t,t + \ell - 1]}^{} = v, \\
	                0, &\text{otherwise.}
	            \end{array}
	        \right.
	    \end{equation*}
	    This leads to
	    \begin{align}
            \lim_{n \to \infty} \frac{1}{n} &\sum_{i = s}^{n + s - 1}
            \mathbbm{1}_{\CZ_{t}^{}(v)}^{}(S^{i}x) \notag\\
            &= \lim_{n \to \infty} \frac{1}{n} \sum_{i = s}^{\ell + m + s - 1}
            \sum_{k = 0}^{\lfloor \frac{n - 1 - i}{\ell + m} \rfloor}
            \mathbbm{1}_{\CZ_{t}^{}(v)}^{} \bigl( S^{i + k(\ell + m)}x \bigr)
            \notag\\
            &= \lim_{n \to \infty} \frac{1}{\ell + m} \sum_{i = s}^{\ell + m +
            s - 1} \frac{1}{\lfloor \frac{n - 1 - i}{\ell + m} \rfloor + 1}
            \sum_{k = 0}^{\lfloor \frac{n - 1 - i}{\ell + m} \rfloor}
            \mathbbm{1}_{\CZ_{t}^{}(v)}^{} \bigl( S^{i + k(\ell + m)}x \bigr).
            \label{equ:slln_apply}
            \intertext{For $s \le i \le \ell + m + s - 1$, we consider the
                family $\bigl( \mathbbm{1}_{\CZ_{t}^{}(v)}^{}(S^{i + k(\ell +
                m)}x) \bigr)_{k \in \IN_{0}}^{}$ and apply
                Theorem~\ref{thm:slln} to each of the inner sums of
                Eq.~$\eqref{equ:slln_apply}$ separately and appropriately put
                the resulting means together. Thus,
                Eq.~$\eqref{equ:slln_apply}$ is almost surely}
            &= \frac{1}{\ell + m} \sum_{i = s}^{\ell + m + s - 1}
            \IE_{\mu_{m}}\bigl( \mathbbm{1}_{\CZ_{t}^{}(v)}^{}(S^{i}x) \bigr) =
            \IE_{\mu_{m}}\bigl( \mathbbm{1}_{\CZ_{t}^{}(v)}^{}(x) \bigr)
            \notag\\
            &= \int_{\IX_{m}^{}}\mathbbm{1}_{\CZ_{t}^{}(v)}^{} \dif \mu_{m}^{}.
            \notag
	    \end{align}
        Note that the penultimate equality is implied by the Perron--Frobenius
        Theorem and the uniqueness of $\mathbold{R}_{m,\ell}^{}$ stated therein.
	    
	    To finish the proof, we need to extend the presented arguments to an
        arbitrary function in $L^{1}(\IX_{m}^{},\mu_{m}^{})$. We define
	    \begin{equation*}
	        \Gamma \defgl 
	        \Bigl\{
	            \sum_{\CZ \in S} a_{\CZ}^{} \mathbbm{1}_{\CZ}^{} \mid S \subset 
	            \FZ(\IX_{m}^{}) 
	            \text{ finite and } a_{\CZ}^{} \in \IC
	        \Bigr\}
	    \end{equation*}
        as the set of simple functions on the measure space $\bigl(
        \IX_{m}^{},\FB_{m}^{},\mu_{m}^{} \bigr)$. By linearity, the validity of
        Eq.~$\eqref{equ:ergodic}$ for $\mathbbm{1}_{\CZ_{t}^{}(v)}^{}$ extends
        to an arbitrary function in $\Gamma$. Due to the Stone--Weierstraß
        theorem \cite[Thm.~1.4]{lang}, $\Gamma$ is dense in $\CC(\IX_{m}^{})$
        and thus also in $L^{1}(\IX_{m}^{},\mu_{m}^{})$ \cite[Thm.~3.1]{lang}.
        This implies the assertion.
	\end{proof}
	
	\begin{thm} 
	    The measure $\mu_{m}^{}$ is ergodic.
	\end{thm}
	
	\begin{proof}
	    This is an immediate consequence of Proposition~\ref{pro:ergodic}. via an 
	    application of Birkhoff's ergodic theorem.
	\end{proof}
	
	\section{Cut and project}
    
    The geometric realisation of fixed points of elements in $\CN_{m}^{}$ is
    derived from the left PF eigenvector $(\lambda_{m}^{},1)^{T}$ of $M_{m}^{}$
    via the identification of $\letter{a}$ and $\letter{b}$ with intervals of
    lengths $\lambda_{m}^{}$ and $1$ and using the left endpoints as
    coordinates. Each of these realisations is called a \emph{noble means set}
    and is denoted by $\Lambda_{m,i}^{}$. It can be shown \cite[Cor.~5.17 and
    Cor.~5.18]{moll} that all $\Lambda_{m,i}^{}$ can be identified as so-called
    \emph{model sets} $\Theta(W_{m,i}^{})$ with \emph{windows} $W_{m,i}^{}$
    within the cut and project scheme $\mathfrak{C} \defgl
    (\IR,\IR,\CL_{m}^{})$; see Figure~\ref{fig:nms_cps} for a compact
    representation and we refer to \cite[Cha.~7]{bagrimm} for a general
    introduction. The underlying lattice $\CL_{m}^{} \defgl \{(x,x^{\prime})
    \mid x \in \IZ[\lambda_{m}^{}] \}$ is independent of $i$. Note that, for
    the generic cases $0 < i < m$, we find the windows
    \begin{equation} \label{equ:window_1}
        W_{m,i}^{} = i\tau_{m}^{} + [\lambda_{m}^{\prime},1]
        \quad \text{with} \quad
        \tau_{m}^{} \defgl -\frac{1}{m} (\lambda_{m}^{\prime} + 1).
    \end{equation}
    In the singular cases $i \in \{0,m\}$, we get
    \begin{alignat}{3}
        W_{m,0}^{(\letter{a}|\letter{a})} &\defgl [\lambda_{m}^{\prime},1) , & 
        W_{m,0}^{(\letter{a}|\letter{b})}
        &\defgl (\lambda_{m}^{\prime},1], \label{equ:window_2} \\
        \quad W_{m,m\vphantom{0}}^{(\letter{a}|\letter{a})} &\defgl
        (-1,-\lambda_{m}^{\prime}], & \quad 
        W_{m,m\vphantom{0}}^{(\letter{b}|\letter{a})} &\defgl
        [-1,-\lambda_{m}^{\prime}), \label{equ:window_3}
    \end{alignat}
    distinguished according to the legal two-letter seeds. In the randomised 
    situation, we consider the geometric realisation of generating random noble 
    means words and study the same cut and project scheme $\mathfrak{C}$ as in the 
    deterministic cases. In this context, we find the following result.

    \begin{figure}[t]
        \centering
        \begin{tikzpicture}[description/.style={fill=white,inner sep=7pt}]
        \matrix (D) [matrix of math nodes, row sep=2em,column sep=4.5em, text
        height=1.5ex, text depth=0.2ex,minimum width=2em,nodes in empty cells]
        { 
            \makebox[1cm]{$\IR$} & \makebox[1cm]{$\IR \times \IR$} &
            \makebox[1cm]{$\IR$}              \\
            \makebox[1cm]{$\IZ[\lambda_{m}^{}]$} & 
            \makebox[1cm]{$\mathcal{L}_{m}^{}$}  &
            \makebox[1cm]{$\IZ[\lambda_{m}^{}]$} \\ 
            \makebox[1cm]{$L$} &  & \makebox[1cm]{$L^{\star}$} \\ };
        \path[-stealth]
        (D-1-2) edge node[above] {$\pi_{1}^{}$} (D-1-1)
        (D-1-2) edge node[above] {$\pi_{2}^{}$} (D-1-3)
        (D-2-2) edge node[above] {$1-1$} (D-2-1)
        (D-2-2) edge node[above] {$1-1$} (D-2-3)
        (D-3-1) edge node[above] {$\star$} (D-3-3);
        \path[-] 
        (D-3-1) edge[double] (D-2-1)
        (D-3-3) edge[double] (D-2-3);
        \path[] 
        (D-2-1) edge node[description] {\rotatebox{90}{$\subset$}} (D-1-1)
        (D-2-2) edge node[description] {\rotatebox{90}{$\subset$}} (D-1-2)
        (D-2-3) edge node[description] {\rotatebox{90}{$\subset$}} (D-1-3);
        \node (1) at ( 3.8,0.7) {{\small dense}};
        \node (2) at (-3.8,0.7) {{\small dense}};
        \end{tikzpicture}
        \caption{Cut and project scheme for the noble means sets $\Lambda_{m,i}^{}$.}
        \label{fig:nms_cps}
    \end{figure}
	
	\begin{pro} \label{pro:windows_random}
        Let $\Lambda_{m}^{}$ be a generating random noble means set.  Then,
        $\Lambda_{m}^{} \subset \Theta\bigl( W_{m}^{} \bigr)$ with $W_{m}^{}
        \defgl [\lambda_{m}^{\prime}-1,1-\lambda_{m}^{\prime}]$.
	\end{pro}
	\begin{proof}
        Assume there is a set $W_{m}^{} = A \cup B$ in the internal space with
        the property $\Lambda_{m}^{} \subset \Theta(W_{m}^{}) = \Theta(A) \cup
        \Theta(B)$. Here, the sets $\Theta(A)$ and $\Theta(B)$ denote the left
        endpoints of intervals generated by the letters $\letter{a}$ and
        $\letter{b}$, respectively.  If $\Lambda_{m}^{}$ is a generating random
        noble means set, the same is true for $\zeta_{m}^{}(\Lambda_{m}^{})$,
        and the sought-after sets $\Theta(A)$ and $\Theta(B)$ are invariant
        under $\zeta_{m}^{}$.  Now, consider $x \in \Lambda_{m}^{}$ and note
        that the interval $[0,x]$ is always mapped to the interval
        $\lambda_{m}^{} \cdot [0,x]$. The sets $\Theta(A)$ and $\Theta(B)$ are
        consequently invariant under $\zeta_{m}^{}$ if and only if for all $0
        \le i \le m$ the inclusions
	    \begin{equation*}
            \zeta_{m,i}^{}\bigl( \Theta(A) \bigr) \subset \Theta(A) \quad
            \text{and} \quad \zeta_{m,i}^{}\bigl( \Theta(B) \bigr) \subset
            \Theta(B)
	    \end{equation*}
        hold. As conditions in the physical space, we get for $0 \le i \le m$
        the $m+1$ systems
	    \begin{align*}
	        \Theta(A) &\supset
	        \biggl\{ 
	            \bigcup_{j=0}^{i-1}\lambda_{m}^{} \Theta(A) + j\lambda_{m}^{}
	        \biggr\} 
	        \cup \lambda_{m}^{} \Theta(B) \cup 
	        \biggl\{
	            \bigcup_{j=i}^{m-1}\lambda_{m}^{} \Theta(A) + j\lambda_{m}^{}
	            + 1
	        \biggr\} \\
	        \Theta(B) &\supset \lambda_{m}^{} \Theta(A) + i\lambda_{m}^{} 
	    \end{align*}
	    and in the internal space the corresponding conjugate systems
	    \begin{equation} \label{equ:superwindow_conditions}
	        \begin{aligned}
	            A &\supset 
	            \biggl\{ 
	                \bigcup_{j=0}^{i-1}\lambda^{\prime}_{m} A +
	                j\lambda^{\prime}_{m} 
	            \biggr\} 
	            \cup \lambda_{m}^{\prime} B \cup 
	            \biggl\{
	                \bigcup_{j=i}^{m-1}\lambda^{\prime}_{m} A +
	                j\lambda^{\prime}_{m} + 1 
	            \biggr\} \\
	            B &\supset \lambda^{\prime}_{m} A + i\lambda^{\prime}_{m} 
	            .
	        \end{aligned}
	    \end{equation}
        As only affine maps appear in Eq.~$\eqref{equ:superwindow_conditions}$,
        it suffices to investigate the extremal cases $i=0$ and $i=m$.
        Furthermore, we can assume that $A$ and $B$ are closed intervals,
        because if $C \in \{A,B\}$ satisfies all conditions of
        Eq.~$\eqref{equ:superwindow_conditions}$ and is no interval, then
        define $\overline{C} = [\inf C,\sup C]$. As all involved maps are
        affine, $\overline{C}$ also meets these conditions and we may define $A
        \defgl [\alpha,\beta]$ and $B \defgl [\gamma,\delta]$. Among the
        remaining conditions of Eq.~$\eqref{equ:superwindow_conditions}$, only
        the following six are not redundant:
	    \begin{equation*}
	        \begin{array}{lll}
	            (1) \; \lambda_{m}^{\prime}\bigl( \beta+(m-1) \bigr) \ge \alpha &
	            \;(2) \; \lambda_{m}^{\prime}\delta \ge \alpha & \;(3) \;
	            \lambda_{m}^{\prime}\gamma \le \beta \\
	            (4) \; \lambda_{m}^{\prime}(\beta + m) \ge \gamma & \;(5) \;
	            \lambda_{m}^{\prime}\alpha+1 \le \beta & \;(6) \;
	            \lambda_{m}^{\prime}\alpha \le \delta.
	        \end{array}
	    \end{equation*}
	    Because of Eqs.~$\eqref{equ:window_1}$ to $\eqref{equ:window_3}$, we may
        assert the relative position $\gamma < \alpha \le \delta < \beta$ of $A$ to
        $B$. This appears to be a linear optimisation problem, which is not
        uniquely solvable in general. Consequently, we additionally demand that the
        interval $W_{m}^{} = [\gamma,\beta]$ be minimal, which leads to the
        condition $\lambda_{m}^{\prime}(\beta + m) = \gamma$.  This equation
        describes the largest translation to the left and if
        $\lambda_{m}^{\prime}(\beta + m) > \gamma$, the length of $W_{m}^{}$ was
        not minimal. By solving the linear optimisation problem of
        Eq.~$\eqref{equ:superwindow_conditions}$ under consideration of all given
        boundary conditions, we get the intervals
	    \begin{equation*}
	        A = [-1,1-\lambda_{m}^{\prime}], \quad B =
	        [\lambda_{m}^{\prime}-1,-\lambda_{m}^{\prime}] \quad \text{and} \quad 
	        W_{m}^{}
	        = [\lambda_{m}^{\prime}-1,1-\lambda_{m}^{\prime}].
	    \end{equation*}
	    These intervals actually satisfy Eq.~$\eqref{equ:superwindow_conditions}$,
	    because for $i=m$ we get
	    \begin{align*}
	        \biggl\{ 
	            \bigcup_{j=0}^{m-1}\lambda_{m}^{\prime}A +
	            j\lambda_{m}^{\prime} 
	        \biggr\} 
	        \cup \lambda_{m}^{\prime}B
	        &= [-(\lambda_{m}^{\prime})^2 +
	           m\lambda_{m}^{\prime},-\lambda_{m}^{\prime}] \\
	        &\qquad\cup 
               [-1-m\lambda_{m}^{\prime},1+(m-1)\lambda_{m}^{\prime}] \\
	        &= [-1,-\lambda_{m}^{\prime}] \cup
	           [-1-m\lambda_{m}^{\prime},1+(m-1)\lambda_{m}^{\prime}] \\
	        &  \subset [-1,1-\lambda_{m}^{\prime}] = A
	        \intertext{and}
	        \lambda_{m}^{\prime}A + m\lambda_{m}^{\prime}
	        &= [-(\lambda_{m}^{\prime})^2 +
	           (m+1)\lambda_{m}^{\prime},(m-1)\lambda_{m}^{\prime}] \\
	        &= [\lambda_{m}^{\prime}-1,(m-1)\lambda_{m}^{\prime}] \\
	        &  \subset [\lambda_{m}^{\prime}-1,-\lambda_{m}^{\prime}] = B.
	    \end{align*}
	    Analogously, we get the corresponding inclusions for $i=0$. Furthermore,
	    the minimality condition of $W_{m}^{}$ is fulfilled because
	    \begin{equation*}
            \lambda_{m}^{\prime}(\beta+m) =
            \lambda_{m}^{\prime}(1-\lambda_{m}^{\prime}+m) =
            \lambda_{m}^{\prime}-1 = \gamma. \qedhere
	    \end{equation*}
	\end{proof}
    
    Henceforth, we indicate the \emph{continuous random noble means hull} by 
    $\IY_{m}^{}$ and denote any element in $\IY_{m}^{}$ as a \emph{random noble 
    means set}. We refer to \cite[Cha.~5]{moll} for a broader overview in this 
    regard.
	
	\begin{thm} \label{thm:meyer}
	    Each random noble means set $\Lambda \in \IY_{m}^{}$ is Meyer.
	\end{thm}
	
	\begin{proof}
        Let $\Lambda_{m}^{}$ be a generating random noble means set. Evidently,
        $\Lambda_{m}^{}$ is relatively dense in $\IR$ with covering radius
        $\lambda_{m}^{}/2$ and, by Proposition~\ref{pro:windows_random}, it is
        a subset of the model set $\Theta\bigl( [\lambda_{m}^{\prime} - 1, 1 -
        \lambda_{m}^{\prime}] \bigr)$. The Meyer property of $\Lambda_{m}^{}$
        then follows from \cite[Thm.~9.1]{moo}. We know that there is a
        generating random noble means set whose orbit is dense,
        $\Lambda_{m}^{}$ say. Now, choose an arbitrary random noble means set
        $\Lambda \in \IY_{m}^{}$ and a converging sequence $(t_{n}^{} +
        \Lambda_{m}^{})_{n \in \IN}^{}$ with limit $\Lambda$. For any $n \in
        \IN$, we find
	    \begin{equation*}
            (t_{n}^{} + \Lambda_{m}^{}) - (t_{n}^{} + \Lambda_{m}^{}) =
            \Lambda_{m}^{} - \Lambda_{m}^{}
	    \end{equation*}
        and therefore $\Lambda - \Lambda \subset \Lambda_{m}^{} -
        \Lambda_{m}^{}$ which means that $\Lambda$ is uniformly discrete. As
        the relative denseness of $\Lambda$ is clear, this proves the
        assertion.
	\end{proof}
	
	\section{Diffraction measure}
    
    In this last section, we present some results concerning the spectral
    nature of the diffraction measure of typical random noble means sets. We
    refer to \cite[Chs.~8~and~9]{bagrimm} for a detailed and readable
    introduction to diffraction theory of model sets; compare \cite{bagrimm}.
    
    To begin with, we briefly discuss the deterministic cases of $\CN_{m}^{}$
    that can be treated with results from the general theory.
	
	\begin{lem}\label{cor:diffr_det}
        For an arbitrary but fixed $m \in \IN$ and $0 \le i \le m$, the
        diffraction measure of $\Lambda_{m,i}^{}$ is a positive and positive
        definite, translation bounded, pure point measure. It is explicitly
        given by
	    \begin{equation} \label{equ:diffr_det}
            \widehat{\gamma_{\Lambda_{m,i}}^{}} = \sum_{k \in
            \CL_{m}^{\circledast}} \lvert A_{m,i}^{}(k) \rvert^{2} \,
            \delta_{k}^{},
	    \end{equation}
	    with the amplitudes
	    \begin{equation*}
	        A_{m,i}^{}(k) = \dens(\Lambda_{m,i}^{}) \e^{-\pi\II k^{\star}
	        (\lambda_{m}^{\prime} + 1)(1-2i/m)} \sinc\bigl( \pi k^{\star}
	        (1-\lambda_{m}^{\prime}) \bigr).
	    \end{equation*}
	\end{lem}
	\begin{proof}
	    To begin with, we note that the Fourier transform of the characteristic
	    function of an interval $[a,b] \subset \IR$ can be represented as
	    \begin{equation} \label{equ:transform_char_func}
	        \widehat{\mathbbm{1}_{[a,b]}^{}}(x) = (b-a) \e^{-\pi\II x(a+b)}
	        \sinc\bigl( \pi x (b-a) \bigr),
	    \end{equation}
	    where $\sinc(z) \defgl \sin(z)/z$. A short calculation based on
        \cite[Thm.~1]{schlottmann} yields $\dens(\Lambda_{m,i}^{}) = (1 -
        \lambda_{m}^{\prime}) / \sqrt{m^{2} + 4}$. Combining this with
        \cite[Thm.~9.4]{bagrimm} and Eqs.~$\eqref{equ:window_1}$ to
        $\eqref{equ:window_3}$, we find
	    \begin{align*}
	        A_{m,i}^{}(k) &= \frac{\dens(\Lambda_{m,i}^{})}{\vol(W_{m,i}^{})}
	        \widehat{\mathbbm{1}_{W_{m,i}^{}}} (-k^{\star}) \\
	        &= \frac{(1 - \lambda_{m}^{\prime}) \e^{-\pi\II k^{\star}
	        (\lambda_{m}^{\prime} + 1)(1 - 2i/m)} \sinc\bigl( \pi k^{\star} (1 -
	        \lambda_{m}^{\prime}) \bigr)}{\sqrt{m^{2}+4}}\\ 
	        &= \dens(\Lambda_{m,i}^{}) \e^{-\pi\II k^{\star} (\lambda_{m}^{\prime} +
	        1)(1 - 2i/m)} \sinc\bigl( \pi k^{\star} (1 - \lambda_{m}^{\prime})
	        \bigr),
	    \end{align*}
	    by an application of Eq.~$\eqref{equ:transform_char_func}$.
	\end{proof}
    \medskip
	In the stochastic situation, we first have to take a closer look at the 
	\emph{autocorrelation} $\gamma_{\Lambda,m}^{}$ of any $\Lambda \in \IY_{m}^{}$, 
	which is defined by
    \begin{equation*}
        \gamma_{\Lambda,m}^{} \defgl \delta_{\Lambda}^{} \circledast 
        \widetilde{\delta_{\Lambda}^{}} \defgl \lim_{n \to \infty} 
        \frac{\delta_{\Lambda_{n}} \ast 
        \widetilde{\delta_{\Lambda_{n}}}}{\vol(B_{n})}
        \quad \text{with} \quad
        \Lambda_{n}^{} \defgl B_{n}^{}(0) \cap \Lambda.
    \end{equation*} 
    Via regularisation of $\delta_{\Lambda}^{}$ and an application of the ergodic 
    theorem for continuous functions \cite[Thm.~2.14z]{muri}, we find that 
    \begin{equation*}
        \gamma_{\Lambda,m}^{} = \IE_{\nu_{m}}^{}\bigl(\delta_{\Lambda}^{} 
        \circledast 
        \widetilde{\delta_{\Lambda}^{}} \bigr) 
    \end{equation*}    
    with $\nu_{m}^{}$ the measure induced by \emph{suspension}
    (\cite[Cha.~11]{cornfeld} and \cite[Sec.~6.1]{moll}) of $\mu_{m}^{}$. Here,
    $\gamma_{\Lambda,m}^{}$ is positive definite by construction and its Fourier
    transform exists due to \cite[Sec.~4]{berg}. We find
	\begin{align} 
	    \widehat{\gamma_{\Lambda,m}^{}} 
        &= \bigl( \IE_{\nu_{m}}^{}(\delta_{\Lambda}^{} \circledast
        \widetilde{\delta_{\Lambda}^{}}) \bigr)^{\widehat{}} = \lim_{n \to
        \infty} \IE_{\nu_{m}}^{}\Bigl(\frac{1}{\vol(B_{n}^{})}
        \widehat{\delta_{\Lambda_{n}^{}}^{}}
        \widehat{\widetilde{\delta_{\Lambda_{n}^{}}^{}}} \Bigr) \notag\\ 
        &= \lim_{n \to \infty} \frac{1}{\vol(B_{n}^{})} \IE_{\nu_{m}}^{}\Bigl(
        \widehat{\delta_{\Lambda_{n}^{}}^{}}
        \overline{\widehat{\delta_{\Lambda_{n}^{}}^{}}} \Bigr) = \lim_{n \to
        \infty} \frac{1}{\vol(B_{n}^{})} \IE_{\nu_{m}}^{}(\lvert X_{n}^{}
        \rvert^{2}) \notag \\
        &= \lim_{n \to \infty} \frac{1}{\vol(B_{n}^{})} \lvert
        \IE_{\nu_{m}}^{}(X_{n}^{}) \rvert^{2} + \lim_{n \to \infty}
        \frac{1}{\vol(B_{n}^{})} \bigl( \IE_{\nu_{m}}^{}(\lvert X_{n}^{}
        \rvert^{2}) - \lvert \IE_{\nu_{m}}^{}(X_{n}^{}) \rvert^{2} \bigr)
        \notag\\
        &= \lim_{n \to \infty} \frac{1}{\vol(B_{n}^{})} \lvert
        \IE_{\nu_{m}}^{}(X_{n}^{}) \rvert^{2} + \lim_{n \to \infty}
        \frac{1}{\vol(B_{n}^{})}
        \IV_{\nu_{m}}^{}(X_{n}^{}),\label{equ:diffr_rek}
	\end{align}
    where $\IV_{\nu_{m}}^{}(X_{n}^{})$ is the variance of \[X_{n}^{}(k) \defgl
    \sum_{x \in \Lambda_{n}} \e^{-2\pi\II kx} = \sum_{x \in \Lambda_{n}}
    \widehat{\delta_{x}^{}},\] provided that all limits exist. The idea of breaking
    up $\widehat{\gamma_{\Lambda,m}^{}}$ according to first and second moments
    will result in $\lim_{n \to \infty} \lvert
    \IE_{\nu_{m}}^{}(X_{n}^{}) \rvert^{2} / \vol(B_{n}^{})$ containing the pure
    point part and $\lim_{n \to \infty}
    \IV_{\nu_{m}}^{}(X_{n}^{})/\vol(B_{n}^{})$ being the absolutely continuous part
    of $\widehat{\gamma_{\Lambda,m}^{}}$. In the following, we will restrict to the
    special case of $m=1$ and consider suitable subsequences to ensure the
    convergence in Eq.~$\eqref{equ:diffr_rek}$. The general case of $m \in \IN$ can
    be treated similarly.
    
    For $n \ge 2$, we define the sequence
    \begin{equation*}
        L_{n}^{} \defgl L_{n-1}^{} + L_{n-2}^{} \quad \text{with} \quad 
        L_{0}^{} \defgl 1 \quad \text{and} \quad
        L_{1}^{} \defgl \lambda_{1}^{}
    \end{equation*}
    that possesses the closed form $L_{n}^{} = \lambda_{1}^{n}$ for any $n \in \IN$
    and furthermore, we set
    \begin{equation}\label{equ:random_var}
        X_{n}^{}(k) \defgl
        \left\{
            \begin{array}{ll}
                X_{n-2}^{}(k) + \e^{-2\pi\II k L_{n-2}^{}} X_{n-1}^{}(k),
                &\text{with probability }p_{0}^{}, \\
                X_{n-1}^{}(k) + \e^{-2\pi\II k L_{n-1}^{}} X_{n-2}^{}(k),
                &\text{with probability }p_{1}^{},
            \end{array}
        \right.
    \end{equation}
    where $X_{0}^{}(k) \defgl \e^{-2\pi\II k}$ and $X_{1}^{}(k) \defgl
    \e^{-2\pi\II k\lambda_{1}^{}}$. Moreover, we define the sequences 
    \begin{equation} \label{equ:con_subseqs}
        \bigl( 
            \CP_{n}^{}
        \bigr)_{n \in \IN_{0}}^{} \defgl 
        \Bigl( 
            \frac{1}{L_{n}^{}} \lvert \IE(X_{n}^{}) \rvert^{2} 
        \Bigr)_{n \in \IN_{0}}^{}
        \quad \text{and} \quad
        \bigl(
            \CS_{n}^{}
        \bigr)_{n \in \IN_{0}}^{} \defgl
        \Bigl(
            \frac{1}{L_{n}^{}} \IV(X_{n}^{})
        \Bigr)_{n \in \IN_{0}},
    \end{equation}
    and derive results on the convergence of $(\CP_{n}^{})_{n \in \IN_{0}}^{}$
    and $(\CS_{n}^{})_{n \in \IN_{0}}^{}$.
	
    We proceed with the derivation of recursion formulas for
    $\IE\bigl(X_{n}^{}(k)\bigr)$ and $\IV\bigl(X_{n}^{}(k)\bigr)$. For the sake
    of readability, we introduce the following abbreviations.
	\begin{equation} \label{equ:abbrevs}
	    \begin{gathered}
	        \e_{n}^{} \defgl \e^{-2\pi\II kL_{n}^{}}, \quad \cos_{n}^{} \defgl
	        \cos(2\pi kL_{n}^{}), \quad X_{n}^{} \defgl X_{n}^{}(k), \\ \IE_{n}^{}
	        \defgl \IE\bigl( X_{n}^{}(k) \bigr) \quad \text{and} \quad \IV_{n}^{}
	        \defgl \mathbb{E}\bigl( \bigl\lvert X_{n}^{}(k) \bigr\rvert^{2} \bigr) -
	        \bigl\lvert \mathbb{E}\bigl( X_{n}^{}(k) \bigr) \bigr\rvert^{2},
	    \end{gathered}
	\end{equation}
	for any $n \in \IN$ and $k \in \IR$. Using the definition of $X_{n}^{}$ in
	Eq.~$\eqref{equ:random_var}$, it is immediate that for $n \ge 2$, we have
	\begin{equation} \label{equ:pp_rek}
	    \begin{aligned}
	        \IE_{n}^{} &= \IE \bigl( p_{0}^{} (X_{n-2}^{} + \e_{n-2}^{}
	        X_{n-1}^{} ) + p_{1}^{} (X_{n-1}^{} + \e_{n-1}^{} X_{n-2}^{}
	        ) \bigr)  \\
	        &= (p_{1}^{} + p_{0}^{} \e_{n-2}^{})\IE_{n-1}^{} + (p_{0}^{} + p_{1}^{}
	        \e_{n-1}^{})\IE_{n-2}^{},
	    \end{aligned}
	\end{equation}
    where $\IE_{0}^{} = \mathrm{e}^{-2\pi\mathrm{i}k}$ and $\IE_{1}^{} =
    \mathrm{e}^{-2\pi\mathrm{i}k\lambda_{1}^{}}$. Firstly, we consider the
    sequence $(\CS_{n}^{})_{n \in \IN_{0}}^{}$. Applying
    Eq.~$\eqref{equ:pp_rek}$ for any $n \ge 2$, we find 
	\begin{align*}
	    \IV_{n}^{} &= \IE 
	    \bigl(
	        p_{0}^{} \lvert X_{n-2}^{} + \e_{n-2}^{} X_{n-1}^{} \rvert^{2} + 
	        p_{1}^{} \lvert X_{n-1}^{} + \e_{n-1}^{} X_{n-2}^{} \rvert^{2}
	    \bigr) - \lvert \IE_{n}^{} \rvert^{2} \\
	    &= \IV_{n-1}^{} + \IV_{n-2}^{} \\
	    &\qquad + 2p_{0}^{}p_{1}^{} \Bigl\{ 
	    \bigl(1 - \cos_{n-2}^{} \bigr) \lvert \IE_{n-1}^{} \rvert^{2}  + \bigl(1 -
	    \cos_{n-1}^{} \bigr) \lvert \IE_{n-2}^{} \rvert^{2} \\
	    &\qquad\qquad -\Rea \bigl[\bigl( 1 - \overline{\e_{n-1}^{}} - \e_{n-2}^{} +\,
	    \overline{\e_{n-1}^{}}\e_{n-2}^{} \bigr) \IE_{n-1}^{} \overline{\IE_{n-2}^{}}
	    \bigr] \Bigr\} \\
	    &\qquad + 2\Rea \bigl[ \bigl(p_{0}^{} \e_{n-2}^{} + p_{1}^{}
	    \overline{\e_{n-1}^{}} \bigr) \bigl( \IE(X_{n-1}^{}
	    \overline{X_{n-2}^{}}) - \IE_{n-1}^{} \overline{\IE_{n-2}^{}} \bigr) \bigr]
	    \tag{$\ast$} \label{hin:indep}\\
	    &= \IV_{n-1}^{} + \IV_{n-2}^{} + 2p_{0}^{}p_{1}^{} \Psi_{n}^{},
	\end{align*}
	with $\IV_{0}^{} = \IV_{1}^{} = 0$ and
	\begin{align}
        \Psi_{n}^{} \defgl \Psi_{n}^{}(k) &\defgl \bigl(1 - \cos_{n-2}^{}
        \bigr) \lvert \IE_{n-1}^{} \rvert^{2} + \bigl(1 - \cos_{n-1}^{} \bigr)
        \lvert \IE_{n-2}^{} \rvert^{2} \notag \\
        &\qquad -\Rea \bigl[\bigl( 1 - \overline{\e_{n-1}^{}} - \e_{n-2}^{} +
        \overline{\e_{n-1}^{}}\e_{n-2}^{} \bigr) \IE_{n-1}^{}
        \overline{\IE_{n-2}^{}} \bigr] \notag\\
        &= \frac{1}{2}\bigl\lvert (1-\e_{n-2}^{}) \IE_{n-1}^{} -
        (1-\e_{n-1}^{}) \IE_{n-2}^{} \bigr\rvert^{2} \ge 0,
        \label{equ:delta_quad}
	\end{align}
    for any $n \ge 2$. We have used that $\IE(X_{n-1}^{} \overline{X_{n-2}^{}})
    - \IE_{n-1}^{} \overline{\IE_{n-2}^{}} = 0$ in ($\ast$) which is a
    consequence of the independence of the random variables $X_{n}^{}$
    \cite[Rem.~6.16]{moll}. Our study of the sequence $(\CS_{n}^{})_{n \in
    \IN_{0}}^{}$ proceeds with some preparing notes on the sequence
    $(\Psi_{n}^{})_{n \ge 2}^{}$; see also Figure~\ref{fig:deltadiff}.
	\begin{lem} \label{lem:delta_func}
        For all $n \ge 2$, the function $\Psi_{n}^{}$ is real analytic.
        Moreover, one has $\Psi_{n}^{}(k) \le 2$ and $\Psi_{n+1}^{}(k) \le
        \Psi_{n}^{}(k)$ for all $k \in \IR$.
	\end{lem}
	\begin{proof}
	    The representation of Eq.~$\eqref{equ:delta_quad}$ immediately shows the
	    analyticity of $\Psi_{n}^{}$ because sums and products of trigonometric
	    functions are real analytic. Next, we observe that
	    \begin{align*}
            \Psi_{2}^{}(k) &= \frac{1}{2}\big\lvert (1-\e^{-2\pi\II k})
            \e^{-2\pi\II k\lambda_{1}^{}} - (1-\e^{-2\pi\II k\lambda_{1}^{}})
            \e^{-2\pi\II k} \big\rvert^{2} \\
	        &= 1 - \cos\bigl(2\pi k(1 - \lambda_{1}^{}) \bigr) \le 2.
	    \end{align*}
        Now, for $n \ge 2$ we define $\psi_{n}^{} \defgl \psi_{n}^{}(k) \defgl
        (1-\e_{n-2}^{})\IE_{n-1}^{} - (1-\e_{n-1}^{})\IE_{n-2}^{}$. Applying
        the recursion for $\IE_{n}^{}$ once on the first summand and using the
        recursion $L_{n}^{} = L_{n-1}^{} + L_{n-2}^{}$ implies
	    \begin{align}\label{equ:rek_delta}
	        \psi_{n+1}^{} &= (1-\e_{n-1}^{})\IE_{n}^{} - (1-\e_{n}^{})\IE_{n-1}^{}
	        %\notag\\
	        %&= (1-\e_{n-1}^{}) \bigl( (p_{1}^{} + p_{0}^{}\e_{n-2}^{})\IE_{n-1}^{}
	        %+ (p_{0}^{} + p_{1}^{}\e_{n-1}^{})\IE_{n-2}^{} \bigr) \notag\\
	        %&\qquad - (1-\e_{n}^{})\IE_{n-1}^{} \notag\\
	        %&= -(p_{0}^{} + p_{1}^{}\e_{n-1}^{}) \bigl((1-\e_{n-2}^{})\IE_{n-1}^{} 
	        %%%- 
	        %(1-\e_{n-1}^{})\IE_{n-2}^{} \bigr) \notag\\
	        = -(p_{0}^{} + p_{1}^{}\e_{n-1}^{}) \psi_{n}^{}.
	    \end{align}
	    This yields the monotonicity of $\Psi_{n}^{}$ because
	    \begin{equation*}
            \lvert \psi_{n+1}^{} \rvert = \lvert p_{0}^{}\psi_{n}^{} +
            p_{1}^{}\e_{n-1}^{} \psi_{n}^{} \rvert \le p_{0}^{} \lvert
            \psi_{n}^{} \rvert + p_{1}^{} \lvert \psi_{n}^{} \rvert = \lvert
            \psi_{n}^{} \rvert,
	    \end{equation*}
	    and therefore
	    \begin{equation*}
            \Psi_{n+1}^{}(k) = \frac{1}{2}\lvert \psi_{n+1}^{}(k) \rvert^{2}
            \le \frac{1}{2}\lvert \psi_{n}^{}(k) \rvert^{2} = \Psi_{n}^{}(k).
            \qedhere
	    \end{equation*}
	\end{proof}
    
    \begin{figure}
       \centering
       \includegraphics[scale=0.6]{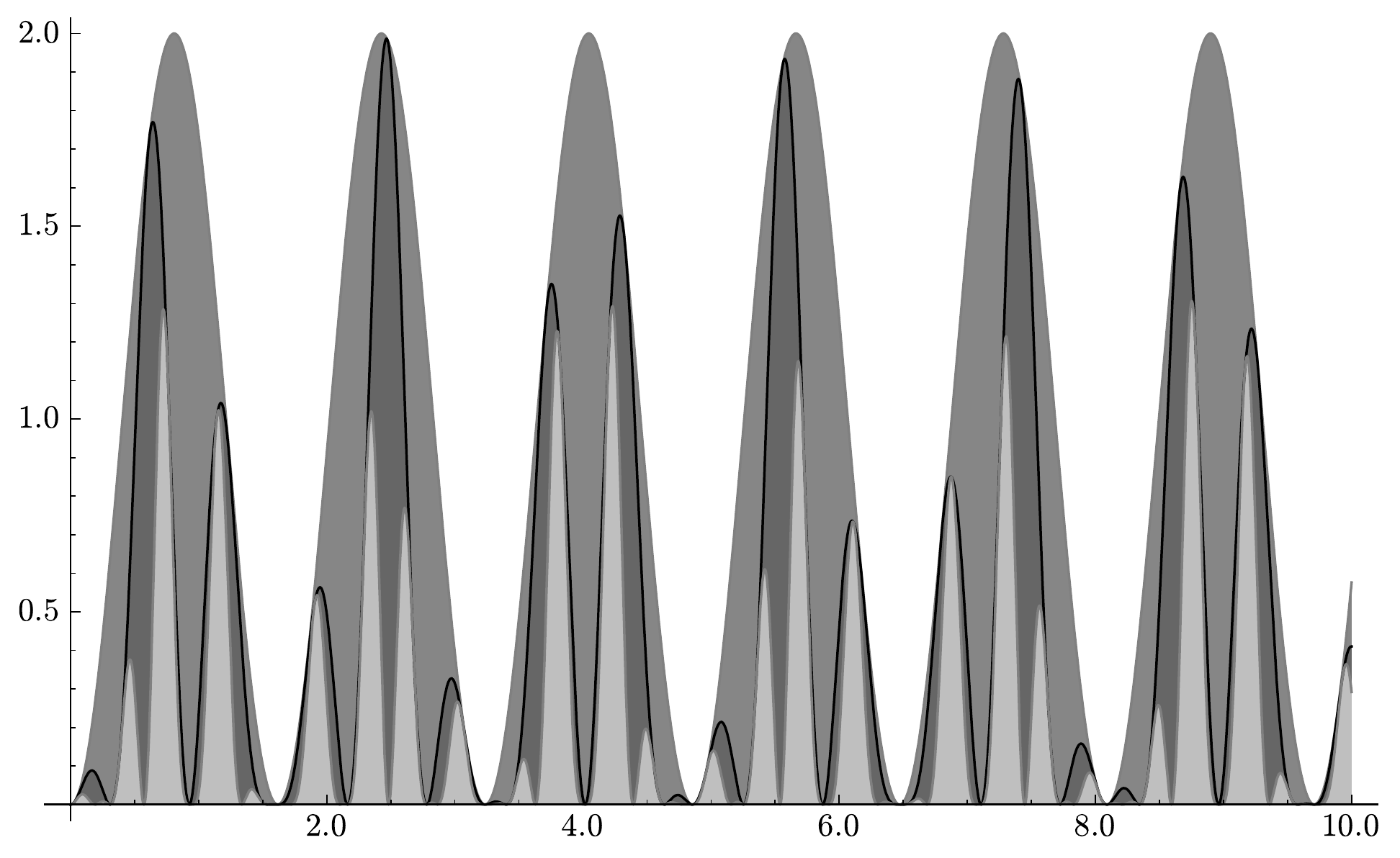}
       \caption{The function $\Psi_{n}^{}$ for $n = 2$ (grey), $n = 3$ (dark
           grey) and $n = 4$ (light grey).}
       \label{fig:deltadiff}
    \end{figure}

    \begin{pro}\label{pro:radon_niko_density}
        For any $n \in \IN_{0}^{}$, consider the function $\phi_{n}^{} \colon
        \IR \longto \IR_{\ge 0}^{}$, defined by
        \begin{equation*}
            \phi_{n}^{}(k) \defgl \frac{1}{L_{n}^{}} \IV\bigl( X_{n}^{}(k) \bigr).
        \end{equation*}
        On $\IR$, the sequence $(\phi_{n}^{})_{n \in \IN_{0}}^{}$ converges
        uniformly to the continuous function $\phi \colon \IR \longto \IR_{\ge
        0}^{}$, with
        \begin{equation}\label{equ:radon_niko_density}
            \phi(k) \defgl \frac{2p_{0}^{}p_{1}^{} \lambda_{1}^{}}{\sqrt{5}} 
            \sum_{i=2}^{\infty} \lambda_{1}^{-i} \Psi_{i}^{}(k).
        \end{equation}
    \end{pro}
    
    \begin{proof}
        From the recursion relation $\IV_{n}^{} = \IV_{n-1}^{} + \IV_{n-2}^{} +
        2p_{0}^{}p_{1}^{} \Psi_{n}^{}$, we conclude the representation
        \begin{equation*}
            \lim_{n \to \infty} \phi_{n}^{}(k) = \lim_{n \to \infty}
            \frac{2p_{0}^{}p_{1}^{}}{L_{n}^{}} \sum_{i=2}^{n} \ell_{1,n+1-i}^{}
            \Psi_{i}^{}(k) = \frac{2p_{0}^{}p_{1}^{} \lambda_{1}^{}}{\sqrt{5}}
            \sum_{i=2}^{\infty} \lambda_{1}^{-i} \Psi_{i}^{}(k),
        \end{equation*}
        where $\ell_{1,n}^{}$ denotes the $n$th Fibonacci number as introduced
        after Eq.~$\eqref{equ:concat_rule}$ on page~\pageref{equ:concat_rule}.
        Next, we observe that $\phi$ is convergent because an application of
        Lemma~\ref{lem:delta_func} yields
        \begin{equation*}
            \phi(k) \le \frac{4p_{0}^{}p_{1}^{} \lambda_{1}^{}}{\sqrt{5}}
            \sum_{i=0}^{\infty} \lambda_{1}^{-i-2} = \frac{4p_{0}^{}p_{1}^{}
            \lambda_{1}^{}}{\sqrt{5}} \le \frac{\lambda_{1}^{}}{\sqrt{5}}.
        \end{equation*}
        Thus, $\phi$ is bounded and the sum consists of non-negative elements
        only.  The uniformity of the convergence is implied by the following
        short calculation
        \begin{align}
            \lvert \phi_{n}^{}(k) - \phi(k) \rvert &= 2p_{0}^{}p_{1}^{} 
            \Bigl\lvert
                \sum_{i=2}^{n} 
                \Bigl(
                    \frac{\ell_{1,n+1-i}^{}}{L_{n}^{}} -
                    \frac{\lambda_{1}^{1-i}}{\sqrt{5}}
                \Bigr) \Psi_{i}^{}(k)
                - \sum_{i=n+1}^{\infty} \frac{\lambda_{1}^{1-i}}{\sqrt{5}} 
                \Psi_{i}^{}(k)
            \Bigr\rvert \notag\\
            &\le 4p_{0}^{}p_{1}^{} 
            \Bigl(
                \Bigl\lvert
                    \frac{(\lambda_{1}^{\prime})^{n-1}}{\lambda_{1}^{n}
                    \sqrt{5}} \sum_{i=0}^{n} (\lambda_{1}^{\prime})^{-i}
                \Bigr\rvert
                + \frac{1}{\lambda_{1}^{n} \sqrt{5}} \sum_{i=0}^{\infty} 
                \lambda_{1}^{-i}
            \Bigr)\notag \\
            &\le 
            \Bigl\lvert
                \frac{(\lambda_{1}^{\prime})^{n-1} -
                1/(\lambda_{1}^{\prime})^{2}}{\lambda_{1}^{n}
                \sqrt{5}(1-1/\lambda_{1}^{\prime})}
            \Bigr\rvert +
            \frac{1}{\lambda_{1}^{n-2}\sqrt{5}}, \label{equ:ac_bound}
        \end{align}
        and both summands in the last line converge to zero, as $n \to \infty$.
        This means that
        \begin{equation*}
            \lim_{n \to \infty} \sup_{k \in \IR} \, \lvert \phi_{n}^{}(k) -
            \phi(k) \rvert = 0,
        \end{equation*}
        which at the same time implies the continuity of $\phi$.
        %    The fact that $\phi$ is real-analytic is implied by the explicit shape 
        %of 
        %    $\Psi_{i}^{}$ presented in Eq.~$\eqref{equ:delta_quad}$. It is given by 
        %the 
        %    absolute square of some exponential series, which gives rise to a sum 
        %of 
        %    trigonometric functions that are all real-analytic themselves. This, 
        %combined 
        %    with the boundedness of all $\Psi_{i}^{}$ proves the assertion.
    \end{proof}
	
	\begin{cor}
	    The roots of $\phi$ are precisely the roots of $\Psi_{2}^{}$, and they are 
	    given by all integer multiples of $\lambda_{1}^{}$.
	\end{cor}
    
    \begin{figure}
       \centering
       \includegraphics[scale=0.6]{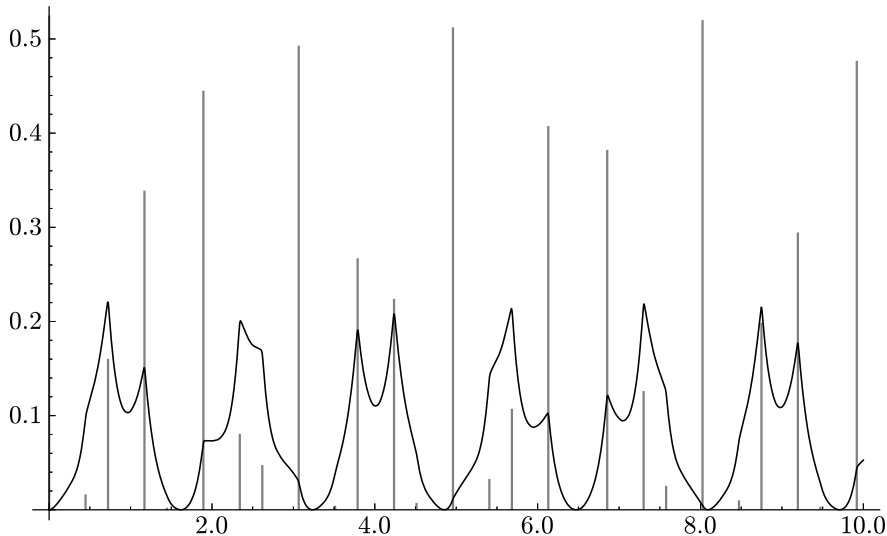}
       \caption{The pure point part (grey) and the absolutely continuous part 
           (black) are illustrated for the case $m = 1$ with $\mathbold{p}_{1}^{} = 
           (1/2,1/2)$.}
       \label{fig:pp_gl}
    \end{figure}
    
    \begin{proof}
        For $n \ge 1$, the recursion formula for $\psi_{n}^{}$ in 
        Eq.~$\eqref{equ:rek_delta}$ can be 
        rewritten as
        \begin{align}
        \psi_{n+1}^{}(k) &= (-1)^{n-1} \psi_{2}^{}(k) \prod_{j=1}^{n-1} 
        \bigl(p_{0}^{} + p_{1}^{}\e^{-2\pi\II k L_{j}^{}} \bigr) \notag \\
        &= (-1)^{n-1} \bigl( \e^{-2\pi\II k\lambda_{1}^{}} -\e^{-2\pi\II k} 
        \bigr) 
        \prod_{j=1}^{n-1} \bigl(p_{0}^{} + p_{1}^{}\e^{-2\pi\II k L_{j}^{}} 
        \bigr). \label{equ:rec_prod}
        \end{align}
        Considering each factor of the product in Eq.~$\eqref{equ:rec_prod}$ 
        separately 
        and including $\Psi_{j}^{}(k) = \lvert \psi_{j}^{}(k) \rvert^{2}/2$ for any 
        $j 
        \ge 2$, we explore
        the function $f_{j}^{} \colon \IR \longto \IR_{\ge 0}^{}$ that is defined as
        \begin{equation*}
        f_{j}^{}(k)
        \defgl \bigl\lvert p_{0}^{} + p_{1}^{}\e^{-2\pi\II k L_{j}^{}} 
        \bigr\rvert^{2} = 
        p_{0}^{2} + p_{1}^{2} + 2p_{0}^{}p_{1}^{} \cos(2\pi k L_{j}^{}).
        \end{equation*}
        Here, for 
        all $j \in
        \IN$, the set of roots of $f_{j}^{}$ reads
        \begin{equation*}
        R_{j}^{} = \biggl\{ \frac{\pm \arccos\bigl(\frac{2p_{0}^{}p_{1}^{} - 
                1}{2p_{0}^{}p_{1}^{}}\bigr) + 2\pi q}{2\pi L_{j}^{}} \,\bigg|\, q 
                \in 
        \IZ \biggr\}.
        \end{equation*}
        Moreover, the expression $\lvert \e^{-2\pi\II k\lambda_{1}^{}}
        -\e^{-2\pi\II k} \rvert^{2} = 2 - 2\cos\bigl( 2\pi k(1-\lambda_{1}^{})
        \bigr)$ vanishes on all $k \in \lambda_{1}^{} \IZ$. This implies that
        \begin{equation*}
            \lambda_{1}^{} \IZ \cup \bigcup_{j = 1}^{n-1} R_{j}^{}
        \end{equation*}
        is the set of roots of $\Psi_{n+1}^{}$ for all $n \ge 1$. Because of 
        Lemma~\ref{lem:delta_func} and 
        the representation
        of $\phi$ in Eq.~$\eqref{equ:radon_niko_density}$, this implies that
        $\lambda_{1}^{}\IZ$ is the set of roots of $\phi$.
    \end{proof}
    \medskip
            \begin{figure}
            \centering
            \includegraphics[scale=0.6]{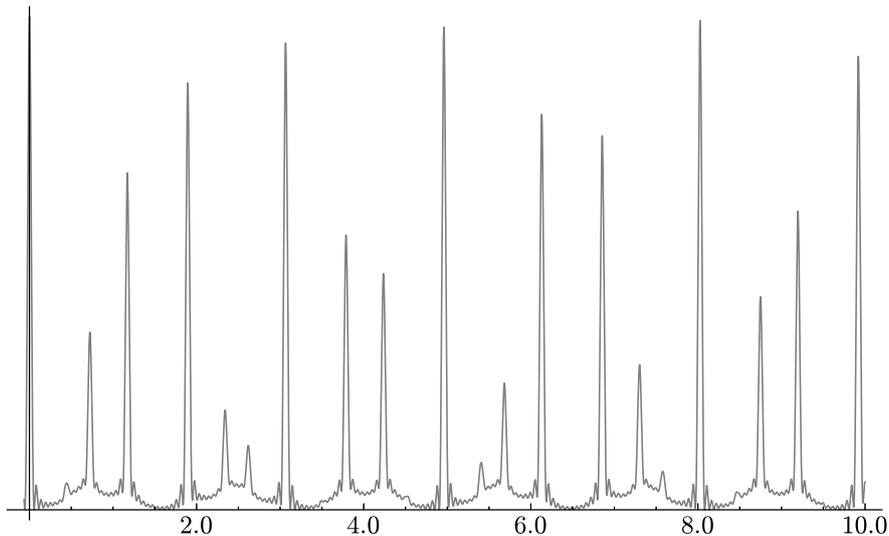}
            \caption{Approximation of the diffraction measure for the case
                $m=1$ with $\mathbold{p}_{1}^{} = (1/2,1/2)$, based on the
                recursion of Eq.~$\eqref{equ:random_var}$ with $n = 6$.}
            \label{fig:fulldiffr}
            \end{figure}
   
	Finally, Proposition~\ref{pro:radon_niko_density} implies the vague convergence
	of the sequence $(\CS_{n}^{})_{n \in \IN_{0}}^{}$ and the existence of
	$\widehat{\gamma_{\Lambda,1}^{}}$ immediately yields the vague convergence of
	$(\CP_{n}^{})_{n \in \IN_{0}}^{}$. Therefore, we almost surely find that
	\begin{equation*}
	    \widehat{\gamma_{\Lambda,1}^{}} = 
	        (\widehat{\gamma_{\Lambda,1}^{}})_{\circleddash}^{} + 
	        (\widehat{\gamma_{\Lambda,1}^{}})_{\mathsf{pp}^{}} 
	    + \phi(k) \lambda,
	\end{equation*}
    where the precise nature of 
    $(\widehat{\gamma_{\Lambda,1}^{}})_{\circleddash}^{}$ stays an open question 
    and needs further study in the future. Following Hof \cite[Thm.~3.2]{hof}, 
    we find
	\begin{equation*}
	    \widehat{\gamma_{\Lambda,1}^{}}(\{k\}) = \lim_{n \to \infty} \frac{1}{L_{n}^{2}}
	    \bigl\lvert \mathbb{E} \bigl(X_{n}^{}(k) \bigr) \bigr\rvert^{2},
	\end{equation*}
	and a sketch of $\widehat{\gamma_{\Lambda,1}^{}}(\{k\})$ and
	$\widehat{\gamma_{\Lambda,1}^{}}$ is illustrated in Figures~\ref{fig:pp_gl}
	and~\ref{fig:fulldiffr}, respectively.
    
    \section*{Outlook}
    
    This paper establishes a first systematic step into the realm of local
    mixtures of substitution rules. The choice of the noble means example
    promised some technical simplifications because all members of $\CN_{m}^{}$
    define the same two-sided discrete hull. One obvious extension of the RNMS
    case can be found in the local mixture of families that do no longer share
    this property. Concerning the computation of the topological entropy, this
    has recently been done for some case  by Nilsson \cite{nilsson2}. More
    generally, one may raise the question which properties a family of
    substitutions must have in order to preserve the features that were derived
    in this treatment.
    
    Leaving the realm of symbolic dynamics and one-dimensional inflation rules,
    one significant enhancement of the theory would be a two or
    three-dimensional example. The (locally) random Penrose tiling was already
    discussed by Godrèche and Luck \cite[Sec. 5.2]{luck}, although a deeper
    mathematical analysis is desirable here, too.

    \section*{Acknowledgements}

    The author wishes to thank Michael Baake, Tobias Jakobi and Johan Nilsson
    for helpful discussions. This work is supported by the German Research
    Foundation (DFG) via the Collaborative Research Centre (CRC 701) through
    the faculty of Mathematics of Bielefeld University.

\end{document}